\documentclass[11pt,letterpaper]{article}

\usepackage{amssymb,amsfonts,amscd,amsthm}
\usepackage[left=0.9in,top=0.9in,right=1in,bottom=0.9in]{geometry}
\usepackage[all,arc]{xy}
\usepackage{enumerate, bbm}
\usepackage{mathrsfs}
\usepackage{tikz}
\usepackage{mathtools}
\usepackage{hyperref}
\usepackage{color}%colors
\usepackage{graphicx}
\usepackage{subfigure,comment}
\usepackage{enumitem} %Allows resuming enumeration and numbering by section
\graphicspath{ {TexImages/} }

\newtheorem{thm}{Theorem}[section]
\newtheorem{cor}[thm]{Corollary}
\newtheorem{prop}[thm]{Proposition}
\newtheorem{lem}[thm]{Lemma}

\theoremstyle{definition}

\newtheorem*{rem*}{Remark}
\newtheorem{conj}[thm]{Conjecture}

\hypersetup{
	colorlinks,
	citecolor=blue,
	filecolor=blue,
	linkcolor=blue,
	urlcolor=blue,
	linktocpage
}

\setenumerate[1]{label=\thesection.\arabic*.} %These two make your enumerated lists start with the section number
\setenumerate[2]{label*=\arabic*.} %Makes subparts labeled by numbers instead of letters.
\setlist[enumerate]{itemsep=2ex, topsep=2ex} %spaces out enumerate/itemize better
\setlist[itemize]{itemsep=2ex, topsep=2ex}

\newcommand{\E}{\mathbb{E}}

\newcommand{\ep}{\epsilon}
\newcommand{\lam}{\lambda}

\newcommand{\Del}{\triangle}

\renewcommand{\l}{\left}
\renewcommand{\r}{\right}
\newcommand{\half}{\frac{1}{2}}

\newcommand{\sm}{\setminus}
\newcommand{\sub}{\subseteq}

\renewcommand{\c}[1]{\mathcal{#1}}

\newcommand{\tr}[1]{\textrm{#1}}
\newcommand{\rec}[1]{\frac{1}{#1}}
\newcommand{\f}[2]{\frac{#1}{#2}}

\newcommand{\mr}[1]{\mathrm{#1}}

\newcommand{\ex}{\mr{ex}}

\title{Triangle-free Subgraphs of Hypergraphs}
\author{Jiaxi Nie \footnote{Dept.\ of Mathematics, UCSD {\tt jin019@ucsd.edu}}\and Sam Spiro\footnote{Dept.\ of Mathematics, UCSD {\tt sspiro@ucsd.edu}. This material is based upon work supported by the National Science Foundation Graduate Research Fellowship under Grant No. DGE-1650112.} \and  Jacques Verstra\"ete \footnote{Dept.\ of Mathematics, UCSD {\tt jverstra@math.ucsd.edu}}}
\date{\today}
\begin{document}
	\maketitle
	
	\begin{abstract}
In this paper, we consider an analog of the well-studied extremal problem for triangle-free subgraphs of graphs for uniform hypergraphs. A {\em loose triangle} is
a hypergraph $T$ consisting of three edges $e,f$ and $g$ such that $|e \cap f| = |f \cap g| = |g \cap e| = 1$ and $e \cap f \cap g = \emptyset$. We prove that if $H$ is an $n$-vertex $r$-uniform hypergraph with maximum degree $\triangle$, then as $\triangle \rightarrow \infty$, the number of edges in a densest $T$-free subhypergraph of $H$ is at least
\[ \frac{e(H)}{\triangle^{\frac{r-2}{r-1} + o(1)}}.\]
For $r = 3$, this is tight up to the $o(1)$ term in the exponent. We also show that if $H$ is a random $n$-vertex triple system with edge-probability $p$ such that $pn^3\rightarrow\infty$ as $n\rightarrow\infty$, then with high probability as $n \rightarrow \infty$, the number of edges in a densest $T$-free subhypergraph is
\[  \min\Bigl\{(1-o(1))p{n\choose3},p^{\frac{1}{3}}n^{2-o(1)}\Bigr\}.\]
We use the method of containers together with probabilistic methods and a connection to the extremal problem for arithmetic progressions of length three due to Ruzsa and Szemer\'{e}di.
	\end{abstract}
	
	\section{Introduction}

The {\em Tur\'{a}n numbers} for a graph $F$ are the quantities $\ex(n,F)$ denoting the maximum number of edges in an $F$-free $n$-vertex graph. The study
of Tur\'{a}n numbers is a cornerstone of extremal graph theory, going back to Mantel's Theorem~\cite{M} and Tur\'{a}n's Theorem~\cite{Turan}. A more general
problem involves studying $\ex(G,F)$, which is the maximum number of edges in an $F$-free subgraph of a graph $G$.  Some celebrated open problems
are instances of this problem, such as the case when $G$ is the $n$-dimensional hypercube -- see Conlon~\cite{Conlon} for recent results.

\medskip

In the case that $F$ is a triangle,
$\ex(G,F) \geq \frac{1}{2}e(G)$ for every graph $G$, which can be seen by taking a maximum cut of $G$, which is essentially tight. In the case $G = G_{n,p}$, the {\em Erd\H{o}s-R\'{e}nyi
random graph}, $\ex(G,F) \sim \frac{1}{2}p{n \choose 2}$ with high probability provided $p$ is not too small, and furthermore every maximum triangle-free subgraph
is bipartite -- see di Marco and Kahn~\cite{dMK} and also Kohayakawa, \L uczak and R\"{o}dl~\cite{KLR}
and di Marco, Hamm and Kahn~\cite{dMHK} for related stability results. The study of $F$-free subgraphs of random graphs when $F$ has chromatic number at least three is undertaken in seminal papers of Friedgut, R\"{o}dl and Schacht~\cite{FRS}, Conlon and Gowers~\cite{CG}, and Schacht~\cite{Sch}.

\subsection{Triangle-free subgraphs of hypergraphs}

In this paper, we consider a generalization of the problem of determining $\ex(G,F)$ when $F$ is a triangle to uniform hypergraphs. We write {\em $r$-graph} instead of $r$-uniform hypergraph. If $G$ and $F$ are $r$-graphs, then $\ex(G,F)$ denotes the maximum number of edges in an $F$-free subgraph of $G$. A {\em loose triangle} is a hypergraph $T$ consisting of three edges $e,f$ and $g$ such that $|e \cap f| = |f \cap g| = |g \cap e| = 1$ and $e \cap f \cap g = \emptyset$.
We write $T^r$ for the loose $r$-uniform triangle. The Tur\'{a}n
problem for loose triangles in $r$-graphs was essentially solved by Frankl and F\"{u}redi~\cite{FF}, who showed for each $r \geq 3$ that $\ex(n,T^r) = {n - 1 \choose r - 1}$ for $n$ large enough, with equality only for the $r$-graph $S_n^r$ of all $r$-sets containing a single vertex. We remark that the Tur\'{a}n problem for $r$-graphs is notoriously difficult in general, and the asymptotic behavior of $\ex(n,K_t^r)$ is a well-known open problem of Erd\H{o}s~\cite{Erd} -- the celebrated Tur\'{a}n conjecture states $\ex(n,K_4^3) \sim \frac{5}{9}{n \choose 3}$.

\medskip

The extremal problem for loose triangles is closely connected to the extremal problem for three-term arithmetic progressions in sets of integers. Specifically, Ruzsa and Szemer\'{e}di~\cite{RS} made the following connection. If $\Gamma$ is an abelian group
and $A \subseteq \Gamma$, define the tripartite linear triple system $H(A,\Gamma)$ whose parts are equal to $\Gamma$ and where
$(\gamma,\gamma + a,\gamma + 2a)$ is an edge if $a \in A$.  In other words, the edges are three-term progressions whose common difference is in $A$.  One can then see that $H(A,\Gamma)$ has $|A||\Gamma|$ edges and is triangle-free whenever $A$ has no three term arithmetic progression.  Ruzsa and Szemer\'{e}di~\cite{RS} showed that every $n$-vertex triangle-free linear triple system
has $o(n^2)$ edges, and applying this to $H(A,\Gamma)$ one obtains Roth's Theorem~\cite{Roth} that $|A| = o(|\Gamma|)$. A construction of Behrend~\cite{B}
gives in $\mathbb Z/n\mathbb Z$ a set $A$ without three-term progressions of size $n/\exp(O(\sqrt{\log n}))$, and so $H(A,\mathbb Z/n\mathbb Z)$ has $n^{2 - o(1)}$
edges in this case. Erd\H{o}s, Frankl, and R\"{o}dl~\cite{EFR} extended these ideas to $r$-uniform hypergraphs, giving the following.

\begin{thm}[Ruzsa and Szemer\'{e}di~\cite{RS}; Erd\H{o}s, Frankl, and R\"{o}dl~\cite{EFR}]\label{T-RS3}
For all $n$ there exists an $n$-vertex $r$-graph which is linear, loose triangle-free, and which has $n^{2 - o(1)}$ edges as $n\rightarrow\infty$.
\end{thm}

This theorem is an important ingredient for our first theorem, giving a general lower bound on the number of edges in a densest triangle-free subgraphs
of $r$-graphs:
	
\begin{thm}\label{T-Lower}
		Let $r \geq 3$ and let $G$ be an $r$-graph with maximum degree $\triangle$.  Then as $\triangle \rightarrow \infty$,
		\[\ex(G,T^r) \ge \triangle^{-\frac{r-2}{r-1} -o(1)} e(G).\]
\end{thm}

If a positive integer $t$ is chosen so that ${t - 1 \choose r - 1} < \triangle \leq {t \choose r - 1}$ and $t|n$,
then the $n$-vertex $r$-graph $G$ consisting of $n/t$ disjoint copies of a clique $K_t^r$ has maximum degree at most $\triangle$
whereas
\[ \ex(G,T^r) = {t - 1 \choose r-1} \frac{n}{t} = \frac{r}{t} e(G) = O(\triangle^{-\frac{1}{r-1}}) \cdot e(G).\]
Here we used the result of Frankl and F\"{u}redi~\cite{FF} that $S_t^r$ is the extremal $T^r$-free subgraph of $K_t^r$ for $t$ large enough.
Therefore for $r = 3$, Theorem \ref{T-Lower} is sharp up to the $o(1)$ term in the exponent of $\triangle$. For $r \geq 4$, the best construction we have gives the following proposition:
	
\begin{prop}\label{P-DetLower}
	For $r \geq 4$ there exists an $r$-graph $G$ with maximum degree $\triangle$ such that as $\triangle \rightarrow \infty$,
\[ \ex(G,T^r)=O(\triangle^{-\frac{1}{2}}) \cdot e(G).\]
\end{prop}

We leave it as an open problem to determine the smallest $c$ such that $\ex(G,T^r) \ge \triangle^{-c-o(1)} \cdot e(G)$ for
every graph $G$ of maximum degree $\triangle$. We conjecture the following for $r = 3$:

\begin{conj}
For $\triangle \geq 1$, there exists a triple system $G$ with maximum degree $\triangle$ such that as $\triangle \rightarrow \infty$, every $T^3$-free subgraph of $G$
has $o(\triangle^{-1/2}) \cdot e(G)$ edges.
\end{conj}

\subsection{Triangle-free subgraphs of random hypergraphs}

	Our next set of results concern random hosts. To this end, we say that a statement depending on $n$ holds {\em asymptotically almost surely} (abbreviated a.a.s.) if the probability that it holds tends to 1 as $n$ tends to infinity.  Let $G_{n,p}^r$ denote random $r$-graph where edges of $K_n^{r}$ are sampled independently
	with probability $p$.  For the $r=2$ case we simply write $G_{n,p}$.
	
 A central conjecture of Kohayakawa, \L uczak and R\"{o}dl~\cite{KLR} was resolved independently by Conlon and
	Gowers~\cite{CG} and by Schacht~\cite{Sch}, and determines the asymptotic value of $\ex(G_{n,p},F)$ whenever $F$ has chromatic number at least three.
	The situation when $F$ is bipartite is more complicated, partly due to the fact that the order of magnitude of Tur\'{a}n numbers $\ex(n,F)$ is not known in general --
see F\"{u}redi and Simonovits~\cite{FS} for a survey of bipartite Tur\'{a}n problems. The case of even cycles was studied by Kohayakawa, Kreuter and Steger~\cite{KKS} and Morris and Saxton~\cite{MS} and complete bipartite graphs were studied by Morris and Saxton~\cite{MS} and by Balogh and Samotij~\cite{BS}.

\medskip

If $F$ consists of two disjoint $r$-sets, then $\ex(n,F)$ is given by the celebrated Erd\H{o}s-Ko-Rado Theorem~\cite{EKR}, and $\ex(n,F) = {n - 1 \choose r - 1}$.
A number of researchers studied $\ex(G_{n,p}^r,F)$ in this case~\cite{BBM}, with the main question being the smallest value of $p$ such that an extremal
$F$-free subgraph of $G_{n,p}^r$ consists of all $r$-sets on a vertex of maximum degree -- $(1 + o(1))p{n - 1 \choose r - 1}$ edges. The same subgraphs are also $T^r$-free, however the extremal subgraphs
in that case are denser and appear to be more difficult to describe. Our second main result is as follows:

	\begin{thm}\label{T-Main3plot}
For all $n \geq 2$ and $p = p(n)\le 1$ with $pn^3\rightarrow\infty$ as $n\rightarrow\infty$, there exist a constant $c>0$ such that asymptotically almost surely
\[\min\{(1-o(1))p\binom{n}{3},p^{\frac{1}{3}}n^2e^{-c\sqrt{\log n}}\}\le \ex(G_{n,p}^3,T^3) \le \min\{(1+o(1))p\binom{n}{3},p^{\frac{1}{3}}n^{2+o(1)}\},\]
and more accurately, for any constant $\delta>0$, when $n^{-3/2+\delta}\le p\le n^{-\delta}$, we have 
\[
\ex(G_{n,p}^3,T^3)\le p^{\frac{1}{3}}n^{2}(\log n)^c.
\]
\end{thm}
We believe that perhaps the lower bound is closer to the truth.

\medskip

Since $G_{n,p}$ for $p>n^{-2+o(1)}$ has maximum degree $\triangle \sim p{n - 1 \choose 2}$ asymptotically almost surely, Theorem \ref{T-Lower} only
gives $\ex(G_{n,p}^3,T^3) \geq p^{1/2 - o(1)}n^2$ a.a.s. The upper bound in Theorem \ref{T-Main3plot} employs the method of {\em containers} developed
by Balogh, Morris and Samotij~\cite{BMS} and Saxton and Thomason~\cite{ST}.

We do not have tight bounds for $\ex(G_{n,p}^r,T^r)$ in general for all $p$ and $r\ge4$.  Partial results and conjectures are discussed in the concluding remarks.

\subsection{Counting triangle-free hypergraphs}

Balogh, Narayanan and Samotij~\cite{BNS} showed that the number of triangle-free $n$-vertex $r$-graphs is $2^{\Theta(n^{r - 1})}$ using the method of containers. Note that
a lower bound follows easily by counting all subgraphs of the $r$-graph $S_n^r$ on $n$ vertices consisting of all $r$-sets containing
a fixed vertex. In this section, we adapt the methods to counting triangle-free hypergraphs with a specified number of edges.

\begin{thm}\label{Containers-3}
	Let $N_3(n,m)$ denote the number of $T^3$-free $3$-graphs with $n$ vertices and $m$ edges. Let $\ep(n)$ be a function such that $\f{\ep(n)\log n}{\log\log n}\rightarrow\infty$ as $n\rightarrow\infty$. Let $\delta=\delta(n)$ be a function such that $\ep(n)<\delta<1/2-\ep(n)$ and let $m=n^{2-\delta}$. Then
		$$
		N_3(n,m) \leq\l(\frac{n^2}{m}\r)^{3m + o(m)}.
		$$
	\end{thm}

We note that an analog of Theorem~\ref{Containers-3} for graphs was proven by Balogh and Samotij~\cite{BS}.  The upper bound on $\ex(G_{n,p}^3,T^3)$ in Theorem \ref{T-Main3plot} will follow quickly from the bound on $N_3(n,m)$ in
Theorem \ref{Containers-3} by taking $m = p^{1/3 - o(1)}n^2$.

\section{Proofs of Theorem~\ref{T-Lower} and Proposition~\ref{P-DetLower}}
For graphs, Foucaud, Krivelevich and Perarnau~\cite{FKP} used certain random homomorphisms to obtain good lower bounds on $\ex(G,F)$.  We briefly summarize these ideas. Let $\c{M}(F)$ denote the family of graphs $F'$ with $e(F')=e(F)$ and which can be obtained from $F$ by identifying vertices.  Let $H$ be an $\c{M}(F)$-free graph with many edges, which we will use as a template for our subgraph of $G$.  Specifically, we take a random mapping $\chi:V(G)\to V(H)$ and then constructs a subgraph $G'\sub G$ such that $uv\in E(G')$ if and only if $\chi(u)\chi(v)\in E(H)$ and such that $\chi(u)\chi(v)\ne \chi(u)\chi(w)$ for any other edge $uw\in E(G)$ (that is, we do not keep edges which are incident and map to the same vertex).  It turns out that $G'$ will be $F$-free because $H$ is $\c{M}(F)$-free, and in expectation $G'$ will have many edges provided $H$ does.

For general $r$-graphs, it is not immediately clear how to extend these ideas in such a way that we can both construct a subgraph with many edges and such that the subgraph is $F$-free.  Fortunately for $T^r$ we are able to do this.  In particular, for this case it turns out we can ignore the family $\c{M}(T^r)$ provided our template $r$-graph is linear.  This is where the Ruzsa-Szemeredi construction of Theorem~\ref{T-RS3} plays its crucial role.

\begin{proof}[Proof of Theorem~\ref{T-Lower}]
	Let $t$ be an integer to be determined later.  Let $\chi$ be a random map from $V(G)$ to $[t]$ and $G_t$ the $r$-graph from Theorem~\ref{T-RS3}.  For ease of notation define $\chi(e)=\{\chi(v_1),\ldots,\chi(v_r)\}$ when $e=\{v_1,\ldots,v_r\}$.  Let $G'$ be the subgraph of $G$ which contains the edge $e$ if and only if
	\begin{itemize}
	    \item[(1)]$\chi(e)$ is an edge of $G_t$, and
	    \item[(2)] $\chi(e')\not\subset\chi(e)$ for any $e'\in E(G)$ with $|e\cap e'|=1$.
	\end{itemize}
	
	We claim that $G'$ is $T^r$-free.  Indeed, let $T$ be a $T^r$ of $G'$, say with edges $e_1,e_2,e_3$ and $e_i\cap e_j=\{x_{ij}\}$ for $i\ne j$.  Because $G_t$ is linear, if $e,e'$ are (possibly non-distinct) edges of $G_t$, then $|e\cap e'|$ is either 0, 1, or $3$.  Note that $\chi(e_i),\chi(e_j)$ are edges of $G_t$ by (1).  Because $e_i\cap e_j=\{x_{ij}\}$ for $i\ne j$, $\chi(x_{ij})\in \chi(e_i)\cap \chi(e_j)$, and by (2) the size of this intersection is strictly less than $r$. Thus $\chi(e_i)\cap \chi(e_j)=\{\chi(x_{ij})\}$.  Further, we must have, say, $\chi(x_{ij})\ne \chi(x_{ik})$ for $k\ne i,j$. This is because (1) guarantees that $\chi(x)$ is a distinct element for each $x\in e_i$, so in particular this holds for $x_{ij},x_{ik}\in e_i$.  In total this implies $\chi(e_1),\chi(e_2),\chi(e_3)$ forms a $T^r$ in $G_t$, a contradiction.
	
	We wish to compute how large $e(G')$ is in expectation.  Fix some $e\in E(G)$.  The probability that $e$ satisfies (1) is exactly $e(G_t) r!/t^r$.  Let $\{e_1,\ldots,e_d\}$ be the edges in $E(G)$ with $|e_i\cap e|=1$.  Given that $e$ satisfies (1), the probability that $\chi(e_1)\not \subset \chi(e)$ is exactly $1-(r/t)^{r-1}$.  Note that for any $v\notin e\cup e_1$, the event $\chi(v)\in \chi(e)$ is independent of the event $\chi(e_1)\not\subset \chi(e)$, so we have \[\Pr[\chi(v)\in \chi(e)|\ e\tr{ satisfies }(1),\ \chi(e_1)\not\subset \chi(e)]=\frac{r}{t}.\]  On the other hand, if $v\in e_1\sm e$, then \[\Pr[\chi(v)\in \chi(e)|\ e\tr{ satisfies }(1),\ \chi(e_1)\not\subset \chi(e)]<\frac{r}{t},\] as knowing some subset containing $\chi(v)$ is not contained in $\chi(e)$ makes it less likely that $\chi(v)\in \chi(e)$.  By applying these observations to each vertex of $e_2\sm e$, we conclude that \[\Pr[\chi(e_2)\not \subset \chi(e)|\ e\tr{ satisfies }(1),\ \chi(e_1)\not \subset \chi(e)]\ge 1-\l(\frac{r}{t}\r)^{r-1}.\]  By repeating this logic for each $e_i$, and using that $e(G_t)=t^{2-o(1)}$, we conclude that \[\Pr[e\tr{ satisfies }(1),\ (2)]\ge \f{e(G_t)r!}{t^r}\l(1-\l(\frac{r}{t}\r)^{r-1}\r)^{r\triangle}=  t^{2-r-o(1)}\l(1-\l(\frac{r}{t}\r)^{r-1}\r)^{r\triangle}.\]
	
	By taking $t=r(r\triangle)^{1/(r-1)}$ and using that $(1-x^{-1})^x$ is a decreasing function in $x$, we conclude by linearity of expectation that \[\E[e(G')]\ge \triangle^{-1+\rec{r-1}-o(1)}\cdot e(G).\]
	
	In particular, there exists some $T^r$-free subgraph of $G$ with at least this many edges, giving the desired result.
\end{proof}

We close this section with a proof of Proposition~\ref{P-DetLower}.

	\begin{proof}[Proof of Proposition~\ref{P-DetLower}]
		By R\"{o}dl~\cite{RT}, there exists an $r$-graph $G$ with $\Theta(n^3)$ edges such that every three vertices is contained in at most one edge.  Let $G'$ be a $T^r$-free subgraph of $G$.  Define $G''$ by deleting every edge of $G'$ which contains two vertices that are contained in at most $2r$ edges.  Note that $e(G')-e(G'')\le 2r{n\choose 2}$.
		
		Assume $G''$ contains an edge $e=\{v_1,\ldots,v_r\}$.  Because $v_1,v_2$ are contained in an edge of $G''$, there exist a set $E_{12}\sub E(G')$ of at least $2r+1$ many edges containing $v_1$ and $v_2$.  As $G$ contained at most one edge containing $v_1,\ v_2$, and $v_3$, any $e_{12}\ne e$ in  $E$ does not contain $v_3$.  Because $v_2,v_3$ are contained in an edge of $G''$, there exists a set $E_{23}\sub E(G')$ of at least $2r+1\ge r+1$ edges containing $v_2,v_3$.  Because $G$ contains at most one edge containing $v_2,v_3,u_i$ for any $u_i\in e_{12}$, we conclude that there exists some $e_{23}\in E_{23}$ such that $e_{12}\cap e_{23}=\{v_2\}$.  Similarly we can find some $e_{13}\in E(G')$ such that $v_1,v_3\in e_{13}$ and such that $e_{13}\cap e_{12}=\{v_1\},\ e_{13}\cap e_{23}=\{v_3\}$.  These three edges form a $T^r$ in $G'$, a contradiction.  We conclude that $G''$ contains no edges, and hence $e(G')\le 2r{n\choose 2}$ for any $T^r$-free subgraph $G'$ of $G$.  As $G$ has maximum degree $\Del=\Theta(n^2)$, we conclude that $\ex(G,T^r)=O(n^2)=O(\Del^{-1/2})\cdot e(G)$.
	\end{proof}
	We note that one can replace the $G$ used in the above proof with an appropriate Steiner system to obtain a regular graph which serves as an upper bound.  It has recently been proven by Keevash~\cite{Keevash} and Glock, K{\"u}hn, Lo, and Osthus~\cite{GKLO} that such Steiner systems exist whenever $n$ satisfies certain divisibility conditions and is sufficiently large.
\section{Proof of Theorem~\ref{T-Main3plot}: Lower Bound.}

As noted in the introduction, the bound of Theorem~\ref{T-Lower} is sharp for $r=3$ by considering the disjoint union of cliques, so we can not improve upon this bound in general.  However, we are able to do better when $G$ contains few copies of $T^r$ by using a deletion argument.

\begin{prop}\label{P-Alt}
    Let $R(G)$ denote the number of copies of $T^r$ in the $r$-graph $G$.  Then for any integer $t\ge 1$,
    \[\ex(G,T^r)\ge (e(G) t^{2-r}- R(G) t^{5-3r})e^{-c\sqrt{\log t}}.\]
\end{prop}
\begin{proof}
Let $\chi$ be a random map from $V(G)$ to $[t]$ and $G_t$ the $r$-graph from Theorem~\ref{T-RS3}.  For ease of notation, if $e=\{v_1,\ldots,v_r\}$ we define $\chi(e):=\{\chi(v_1),\ldots,\chi(v_r)\}$.  Let $G'$ be the subgraph of $G$ which contains the edge $e$ if and only if $\chi(e)$ is an edge of $G_t$.
	
We claim that $e_1,e_2,e_3\in E(G')$ form a $T^r$ in $G'$ if and only if $e_1,e_2,e_3$ form a $T^r$ in $G$ and $\chi(e_1)=\chi(e_2)=\chi(e_3)$ is an edge of $G_t$.  Indeed, the backwards direction is clear. Assume for contradiction that these edges form a $T^r$ in $G'$ and that $e_1\ne e_2$.  Let $x_{ij}$ for $i\ne j$ be such that $e_i\cap e_j=\{x_{ij}\}$.  Because $G_t$ is linear, if $e,e'$ are (possibly non-distinct) edges of $G_t$, then $|e\cap e'|$ is either 0, 1, or $r$.  Because each $e_i$ is in $E(G')$, we have $\chi(e_i)\in E(G_t)$ by construction.  In particular, as $e_1\cap e_2=\{x_{12}\}$ and $\chi(e_1)\ne \chi(e_2)$, we must have $\chi(e_1)\cap \chi(e_2)=\{\chi(x_{12})\}$. As $e_3$ contains elements in and not in $e_1$ (namely $x_{13}$ and $x_{23}$), we must have $\chi(e_1)\cap \chi(e_3)=\{\chi(x_{13})\}$.  Similarly we have $\chi(e_2)\cap \chi(e_3)=\{\chi(x_{23})\}$.  Because $\chi(e_i)$ is an $r$-set for each $i$, we have $\chi(x_{ij})\ne \chi(x_{ik})$ for $\{i,j,k\}=\{1,2,3\}$. Thus $\chi(e_1),\chi(e_2),\chi(e_3)$ form a $T^r$ in $G_t$, a contradiction.

Let $G''\sub G'$ be a subgraph obtained by deleting an edge from each $T^r$ of $G'$.  By construction $G''$ is $T^r$-free. We conclude by linearity of expectation that
\begin{align*}
\ex(G,T^r)&\ge \E[e(G'')]\ge \E[e(G')-R(G')]\\
&=\f{e(G_t) r!}{t^r} e(G)-\f{e(G_t) r!}{t^{3r-3}}R(G)\\ 
&\ge (e(G) t^{2-r}- R(G) t^{5-3r})e^{-c\sqrt{\log t}}.
\end{align*}
\end{proof}

\begin{cor}\label{C-rLow}
For any integer $r\ge 3$, and function $p=p(n)\le1$ such that $p^{2/(2r-3)}n\ge2$, we have
    \[\E[\ex(G_{n,p}^r,T^r)]\ge p^{\frac{1}{2r-3}} n^{2}e^{-c\sqrt{n}},\]
for some constant $c>0$.
\end{cor}
\begin{proof}
Note for $n\ge 4$ that $\E[e(G_{n,p}^r)]=p{n\choose r}\ge p n^{r}/(2r)^r\ge pn^r/r^{2r}$, and that $\E[R(G_{n,p}^r)]\le p^3 n^{3r-3}$.  Plugging these into the bound of Proposition~\ref{P-Alt} gives
\[\E[\ex(G_{n,p}^r,T^r)]\ge (p n^r t^{2-r}- p^3 n^{3r-3}t^{5-3r})e^{-c\sqrt{\log t}}.\]
Take $t=p^{2/(2r-3)} n^{1/2}$, we conclude for sufficiently large $n$ that
\[\E[\ex(G_{n,p}^r,T^r)]\ge p^{\frac{1}{2r-3}} n^{2}e^{-c\sqrt{\log n}}.\]
\end{proof}

To get the a.a.s. result of Theorem~\ref{T-Main3plot}, we use Azuma's inequality (See for example in Alon and Spencer~\cite{AS}) applied to the edge exposure martingale.

\begin{lem}\label{L-Azuma}
    Let $f$ be a function on $r$-graphs such that $|f(G)-f(H)|\le 1$ whenever $H$ is $G$ with exactly one edge added or deleted.  Then for any $\lam>0$, \[\Pr\l[|f(G_{n,p}^r)-\E[f(G_{n,p}^r)]|>\lam \sqrt{n\choose r}\r]<e^{-\frac{\lam^2}{2}}.\]
\end{lem}

\begin{proof}[Proof of Theorem~\ref{T-Main3plot}: Lower Bounds]
    Let $\ep(n)=e^{k\sqrt{\log n}}$, where $k>0$ is some large enough constant. For $p\le n^{-3/2}/\ep(n)$, it is not difficult to show that a.a.s. $G_{n,p}^3$ contains $o(pn^3)$ copies of $T^{3}$, and by deleting an edge from each of these loose cycles we see that $\ex(G^3_{n,p},T^3)=(1-o(1)p{n\choose 3}$ a.a.s.

   For $n^{-3/2}/\ep(n)\le p\le n^{-3/2}\ep(n)$, we do an extra round of random sampling on the edges of $G^r_{n,p}$ and keep each edge with probability $p':=\ep(n)^{-2}$. The $r$-graph we obtained is equivalent to $G^r_{n,pp'}$, with $pp'\le n^{-3/2}/\ep(n)$. Thus $\ex(G^3_{n,p},T^3)=(1-o(1))pp'{n\choose3}=(1-o(1))p{n\choose3}/\ep(n)^2)$ a.a.s. Using $p\ge n^{-3/2}/\ep(n)$, we conclude that $\ex(G^3_{n,p},T^3)\ge p^{1/3}n^{2}e^{-c\sqrt{\log n}}$ a.a.s. in this range.

    We now consider $p\ge n^{-3/2}\ep(n)$.  The bound in expectation follows from Corollary~\ref{C-rLow}.  To show that this result holds a.a.s., we observe that $f(G)=\ex(G,T^3)$ satisfies the conditions of Lemma~\ref{L-Azuma}.
    For ease of notation let $X_{n,p}=\ex(G_{n,p}^3,T^3)$ and let  $B_{n,p}=p^{1/3}n^{2}e^{-c\sqrt{\log n}}$ be the lower bound for $\E[X_{n,p}]$ given in Corollary~\ref{C-rLow}.  Setting $\lam=\half B_{n,p} {n\choose 3}^{-1/2}$ and applying Azuma's inequality, we find
    \begin{align*}\Pr\l[X_{n,p}<\f{1}{2}B_{n,p}\r]&\le \Pr\l[X_{n,p}-\E[X_{n,p}]< \lam {n\choose 3}^{\f{1}{2}}\r]\\ &\le \Pr\l[\l|X_{n,p}-\E[X_{n,p}]\r|<\lam{n\choose 3}^{\frac{1}{2}}\r]\le \exp(-\frac{\lam^2}{2}).\end{align*}

    Note that for $p\ge n^{-3/2}\ep(n)$ we have $\lam\ge e^{(k/3-c)\sqrt{\log n}}\rightarrow\infty$ as $n\rightarrow\infty$. So we conclude the a.a.s. result.
\end{proof}

\section{Containers}
The method of containers developed
by Balogh, Morris and Samotij~\cite{BMS} and Saxton and Thomason~\cite{ST} is a powerful technique that has been used to solve a number of combinatorial problems. Roughly, the idea is for a suitable hypergraph $H$ to find a family of sets $\c{C}$ which contain every independent set of $H$, and in such a way that $|\c{C}|$ is small and each $C\in \c{C}$ contains few edges.  For example, by letting $H$ be the 3-uniform hypergraph where each edge is a $K_3$ in some graph $G$, we see that independent sets of $H$ correspond to triangle-free subgraphs of $G$.  The existence of containers then allows us to better understand how these subgraphs of $G$ behave. 

 We proceed with the technical details of this approach. Given an $r$-graph $H=(V,E)$, let $v(H)=|V|$, $e(H)=|E|$, and let $\mathcal{P}(V)$ be the family of subsets of $V$. For $A$ a set of vertices in $H$, let $d(A)$ be the number of edges in $H$ that contain $A$.  Let $\bar{d}(H)$ be the average degree of $H$, and let $\triangle_j(H)=\max_{|A|=j}d(A)$. In order to establish our upper bounds, we need to use the following container lemma for hypergraphs:

    \begin{lem}[Balogh, Morris and Samotij~\cite{BMS}]\label{Contn}
Let $r,b,l\in\mathbb{N}$, $\delta=2^{-r(r+1)}$, and $H=(V,E)$ an $r$-graph such that
    $$
    \triangle_j(H)\leq \l(\frac{b}{v(H)}\r)^{j-1}\frac{e(H)}{l},\ j \in\{1,2,\dots,r\}.
    $$
     Then there exists a collection $\mathcal{C}$ of subsets of\ \ $V$ and a function $f:\mathcal{P}(V)\rightarrow\mathcal{C}$ such that:
    \begin{itemize}
        \item[$(1)$] For every independent set $I$ of $H$, there exists $S\subset I$ with $|S|\leq(k-1)b$ such that $I\subset f(S)$.
        \item[$(2)$] For every $C\in\mathcal{C}$, $|C|\leq v(H)-\delta l$.
    \end{itemize}
\end{lem}

We will use this container lemma to give an upper bound for $N_3(n,m)$, which we recall is the number of $T^3$-free 3-graphs with $n$ vertices and $m$ edges.  The idea is to consider the $3$-graph $H$ with $V(H)=E(K^{r}_n)$ and $E(H)$ consisting of $T^r$ in  $K^{r}_n$. Notice that the container lemma requires upper bounds for the maximum codegrees of the hypergraph. In order to meet this requirement, we will use a balanced-supersaturation lemma for $T^r$.
\begin{lem}[Balogh, Narayanan and Skokan~\cite{BNS}]\label{Saturtation}
   For every $r\geq3$, there exists $c=c(r)$ such that the following holds for all $n$. Given any $r$-graph $G$ on $[n]$ with $e(G)=tn^{r-1}$, $t\geq6(r-1)$, let $S=1$ if $r=3$ and $S=tn^{r-4}$ if $r\ge 4$.  Then there exists a $3$-graph $H$ on $E(G)$, where each edge of $H$ is a copy of $T^r$ in $G$, such that:
 \begin{itemize}
     \item[$(1)$]$\bar{d}(H)\geq c^{-1}t^{3}S^{2}$.
     \item[$(2)$] $\triangle_j(H)\leq ct^{5-2j}S^{3-j}$,for  $j=1,2$.
 \end{itemize}
\end{lem}

Using the previous two lemmas, we derive the following container lemma for $T^3$-free hypergraphs. Similar result for $T^r$-free hypergraphs can also be obtained using the same idea, and we briefly comment on these results in the concluding remarks.
\begin{lem}\label{Containers-C}
 For any integer $n$ and positive number $t$ with $12\leq t\leq \binom{n}{3}/n^2$, there exists a collection $\c{C}$ of subgraphs of $K^{3}_n$ such that for some constant $c$:
 \begin{itemize}
     \item[$(1)$] For any $T^3$-free subgraph $J$ of $K^{3}_n$, there exists $C\in\c{C}$ such that $J\subset C$.
     \item[$(2)$] $|\mathcal{C}|\leq \exp\l(\f{c\log(t)n^2}{\sqrt{t}}\r)$.
     \item[$(3)$] For every $C\in\mathcal{C}$, $e(C)\leq tn^2$.
 \end{itemize}
\end{lem}

\begin{proof}
By Lemma~\ref{Saturtation}, there exists a positive constant $c_1$ such that for any 3-graph $G$ on $[n]$ with $e(G)=t_0n^2$ and $t_0\geq t$, there exists a $3$-graph $H$ on $E(G)$ such that:
\begin{itemize}
    \item[(1)] Every edge of $H$ is a copy of $T^3$.
    \item[(2)]  $\bar{d}(H)\geq c_1^{-1}t_0^{3}$.
    \item[(3)] $\triangle_j(H)\leq c_1t_0^{5-2j},j=1,2$. $\triangle_3(H)=1$.
\end{itemize}

We can then use Lemma~\ref{Contn} on $H$ with $l=t_0n^2/{(3c_1^2)}$ and $b=n^2/\sqrt{c_1t_0}$ to get a collection $\c{C}$ of subgraphs of $G$ such that they contain all $T^3$-free subgraphs of $G$, and for each $C\in \c{C}$, $e(C)\leq (1-\epsilon)t_0n^2$ for some constant $\epsilon>0$. Also, we have
$$
|\c{C}|\leq\sum_{s=1}^{2b}\binom{tn^2}{s}\leq \exp\l(\frac{c_2\log(t_0) n^2}{\sqrt{t_0}}\r)
$$
for some constant $c_2>0$.

We use the above argument on $G=K^{3}_n$ to get a family of containers $\c{C}_1$. Notice that the containers of $\c{C}_1$ are also 3-graphs on $[n]$, so we can repeat this argument on each $C\in \c{C}_1$ with more than $tn^2$ edges to get a new collection of containers $\c{C}_2$.  We do this repeatedly until all containers have less than $tn^2$ edges. Since in each step the number of edges will decrease by a constant $(1-\epsilon)$, this process must stop after at most $\log{(n/t)}/\ep$ steps. For $k\ge 0$, define $t_{k+1}=(1-\epsilon)t_k$.  Let $M$ be the largest integer such that $t_M>t$. Because $t_0\le  \binom{n}{3}/n^2$, the number of containers we have in the end is less than
$$
\begin{aligned}
\prod_{i=0}^M\exp\l(\frac{c_2\log (t_i) n^2}{\sqrt{t_i}}\r)&=\exp\l(\sum_{i=0}^M\frac{c_2\log (t_i) n^2}{\sqrt{t_i}}\r)\\
&\leq\exp\l(\frac{c\log(t) n^2}{\sqrt{t}}\r)
\end{aligned}
$$
for some constant $c>0$.
\end{proof}

With the lemma above, we are ready to bound $N_3(n,m)$.
\begin{proof}[Proof of Theorem~\ref{Containers-3}]
Let $\c{C}$ be a collection of containers and $c$ a constant as in Lemma~\ref{Containers-C} with $t=n^{2\delta +\epsilon_1(n)}$, where $\epsilon_1(n)=\frac{2\log\log{n}}{\log{n}}$. By considering all subgraphs of each $C\in \c{C}$ with $m=n^{2-\delta}$ edges, and recalling that $\ep(n)<\delta<1/2-\ep(n)$ and ${n\choose k}\le (en/k)^k$, we conclude that for some suitable $\ep(n)$, 
 $$
\begin{aligned}
N_3(n,m)&\leq\exp\l(\frac{c\log (t) n^2}{\sqrt{t}}\r)\cdot\binom{tn^2}{m}\\
&\leq\exp\l(c\log t\cdot \f{m}{\log n}+\l(1+\l(3\delta+\epsilon_1(n)\r)\log{n}\r)m\r)\\
&\leq\exp\l(\delta\log {n}\cdot m\l(3+\l(2+o(1)\r)\frac{\log\log{n}}{\delta\log{n}}\r)\r)\\
&\leq \l(\frac{n^2}{m}\r)^{3m+o(m)}.\\
\end{aligned}
$$

\end{proof}

We are now ready to prove the upper bound of Theorem~\ref{T-Main3plot}
\begin{proof}[Proof of Theorem~\ref{T-Main3plot}: Upper Bound]
We will only present the proof of the upper in terms of $o(1)$ for the whole range. The proof of the more accurate upper bound in the smaller range is essentially the same, with more careful and explicit computation for the $o(1)$ factor. For $p\le n^{-3/2+o(1)}$, the proof for the upper bound is exactly the same as that for the lower bound. We now consider $n^{-3/2+\ep(n)}\le p\le n^{-\ep(n)}$ for some small function $\ep(n)=o(1)$. Our goal is to show
$$
\Pr[ex(G^3_{n,p},T^3)\geq m]\rightarrow0,\ as\ n\rightarrow\infty,
$$
for some $m=p^{1/3}n^{2+o(1)}$. Let $X_m$ be the expected number of $T^3$-free subgraphs in $G^3_{n,p}$ with $m$ edges. By Theorem~\ref{Containers-3}, when $n^{3/2+\ep_1(n)}\le m\le n^{2-\ep_1(n)}$ for some function $\ep_1(n)=o(1)$, there exist a function $\epsilon_2(n)=o(1)$ such that the expectation of $X_m$ satisfies
$$
\begin{aligned}
\mathbb{E}[X_m]&=N_3({n,m})\cdot p^m\\
&\leq \l(\frac{n^2}{m}\r)^{m\l(3+\epsilon_2(n)\r)}p^m\\
&=\l(\l(\frac{n^2}{m}\r)^{\l(3+\epsilon_2(n)\r)}p\r)^m.
\end{aligned}
$$
We can let $m=p^{1/3-\epsilon_3(n)}n^2$ for some small function $\epsilon_3(n)=o(1)$ such that
$$
\l(\frac{n^2}{m}\r)^{\l(3+\epsilon_2(n)\r)}p<1.
$$
Also we can pick some suitable $\ep(n)$, so that $n^{3/2+\ep_1(n)}\le m\le n^{2-\ep_1(n)}$. 
Thus we have $\mathbb{E}[X_m]\rightarrow 0$ as $n\rightarrow\infty$. Then by Markov's inequality, we have
$$
\Pr[ex(G^3_{n,p},T^3)\geq m]=\Pr[X_m\geq 1]\leq\mathbb{E}[X_m]\rightarrow0,\ as\ n\rightarrow\infty.
$$
So a.a.s. we have
$$ex(G^3_{n,p},T^3)<m=p^{\f{1}{3}}n^{2+o(1)}.$$
Finally for $p\ge n^{-o(1)}$, we have $\ex(G^3_{n,p},T)<\ex(K^3_n,T)=\Theta(n^2)=p^{1/3}n^{2+o(1)}$ a.a.s.
\end{proof}

	\section{Concluding Remarks}
\begin{itemize}
    \item  We are able to generalize Theorem \ref{Containers-3}
to $r$-graphs as follows:
	
	\begin{thm}\label{Containers-r}
		Let $N_r(n,m)$ denote the number of $T^r$-free $r$-graphs with $n$ vertices and $m$ edges. Let $r\geq 4$, $0 < \delta < 3/2$, and $m=n^{3-\delta}$. Then
		$$
		N_r(n,m) \leq\l(\frac{n^{r-1}}{m}\r)^{\l(1+\frac{2\delta}{3r - 12 + 3\delta}\r)m +o(m)}.
		$$
		\\
		When $r>4$, let $m=n^{3+\delta}$ with $\delta$ some constant satisfying $0<\delta<r-4$.  Then we have
		$$
		N_r(n,m) \leq\l(\frac{n^{r-1}}{m}\r)^{m+o(m)}.
		$$
	\end{thm}
	This bound will also leads to an upper bound for $\ex(G_{n,p}^r,T^r)$ when $n^{-r+3/2+o(1)}\le p\le 1$, which is essentially tight for $p=p(n)$ with $n^{-r+4+o(1)}\le p\le 1$. However, there is a gap between the lower bound and upper bound in the range $n^{-r+3/2+o(1)}\le p\le n^{-r+4+o(1)}$.
    \item Using the same techniques for the $r=3$ case, we are able to show the following.
	\begin{thm}\label{T-Mainrplot}
        For $r\ge 4$ and $0\le x\le r$ a constant, let $p=n^{-r+x}$ and define 
        \[f_r(x)=\lim_{n\rightarrow\infty}\log_n{\E[\ex(G^r_{n,p},T^r)}].\] 
        Then for $0\le x\le 3/2$, $f_r(x)=x$; for $4<x\le r$, $f_r(x)=x-1$; and for $3/2<x\le 4$, we have
        $$
        \max\{\f{x+3r-6}{2r-3},x-1\}\le f_r(x)\le \f{3x+3}{5}.
        $$
    \end{thm}
    The bounds for $x\le 3/2$ come from deleting an edge from each triangle in $G_{n,p}^{r}$.  For $x>3/2$, the upper bound follow from Theorem~\ref{Containers-r}, the first lower bound follows from Corollary~\ref{C-rLow}, and the second lower bound follows from taking every edge containing a given vertex. 
    \item We believe that the upper bound is perhaps closer to the truth and have the following conjecture.
    \begin{conj}\label{Conj-r}
        For $r\ge 4$ and $0\le x\le r$ a constant, let $p=n^{-r+x}$ and $f_r(x)$ as defined in Theorem~\ref{T-Mainrplot}.  Then for $\f{3}{2}<x\le 4$,
        
        \[f_r(x)=\f{3x+3}{5}\]
    \end{conj}
    \item For the deterministic case, we note that one can extend the proof of Theorem~\ref{T-Lower} to other $F$ by defining maps $\chi:V(G)\to V(H)$ for suitable $H$. In this case a second step must be done to effectively bound $\ex(G,F)$.  We plan to do this in a followup paper.
\end{itemize}

\end{document}